\documentclass[12pt]{amsart}

\usepackage{dsfont}
\usepackage{amssymb}
\usepackage{amsmath}
\usepackage{amsthm}
\usepackage{mathrsfs}
\usepackage{stmaryrd}
\usepackage{bbm}
\usepackage{oldgerm}
\usepackage[english]{babel}
\usepackage[T1]{fontenc}
\usepackage[bbgreekl]{mathbbol}

\DeclareSymbolFontAlphabet{\mathbbm}{bbold}

\usepackage[inkscapelatex=false]{svg}

\usepackage[all]{xy}
\usepackage{hyperref}
\usepackage{upref}
\usepackage{xcolor}
\usepackage{tikz-cd}
\usepackage{tikz}

\usepackage{latexsym}
\usepackage{graphicx}
\usepackage[utf8]{inputenc}
\usepackage{eurosym}

\usepackage{footnote}

\usepackage[backend=biber,style=numeric,maxbibnames = 99]{biblatex}

\usepackage{fancyhdr}
\fancypagestyle{plain}{
    \lhead{}
    \fancyhead[R]{\thepage}
    \fancyhead[L]{}
    
    \fancyfoot{}
}

\usepackage{todonotes}
\usepackage{forest}
\usepackage{changepage}
\usepackage[center]{caption}
\usepackage{pgf,pgfplots}
\pgfplotsset{compat=1.15}
\usetikzlibrary{arrows}

\hypersetup{
 colorlinks,
 linkcolor={red!50!black},
 citecolor={blue!50!black},
 urlcolor={blue!80!black}
}

\theoremstyle{plain}
\newtheorem{teo}{Theorem}[section]
\newtheorem{prop}[teo]{Proposition}
\newtheorem{cor}[teo]{Corollary}
\newtheorem{lem}[teo]{Lemma}

\theoremstyle{definition}
\newtheorem{defi}[teo]{Definition}

\theoremstyle{remark}
\newtheorem{rema}[teo]{Remark}

\newtheorem{exemple}[teo]{Example}

\DeclareMathOperator{\map}{map}
\DeclareMathOperator{\Conf}{Conf}
\DeclareMathOperator{\UConf}{UConf}
\DeclareMathOperator{\id}{id}
\DeclareMathOperator{\im}{im}
\DeclareMathOperator{\Id}{Id}
\DeclareMathOperator{\Mor}{Mor}
\DeclareMathOperator{\Ob}{Ob}
\DeclareMathOperator{\Top}{Top}

\DeclareMathOperator{\Moore}{Moore}
\DeclareMathOperator{\Sing}{Sing}

\newcommand{\C}{\mathcal{C}}

\title{The Barratt--Priddy--Quillen theorem via scanning methods}
\author{Marie-Camille Delarue}
\thanks{Université Paris Cité and Sorbonne Université, CNRS, IMJ-PRG, F-75013 Paris, France.}

\begin{document}

\begin{abstract}
    The homology of the symmetric groups stabilizes, and the Barratt--Priddy--Quillen theorem identifies the stable homology with that of the infinite loop space underlying the sphere spectrum.
    We formulate a new proof inspired by Galatius and Randal-Williams using scanning methods.
    We build a topological model for the monoid formed by all the symmetric groups as a category of paths in $\mathbb{R}^\infty$ and build a scanning map from this model to a space of local images.
\end{abstract}

\maketitle
\tableofcontents

\section{Introduction}
The homology groups of the symmetric groups, $H_k(\Sigma_n;\mathbb{Z})$, were computed for every $k$ and $n$ by Nakaoka \cite{nakaoka}, but these computations proved intricate. 
It follows from Nakaoka's computation that the $k$th homology of the symmetric groups $H_k(\Sigma_n;\mathbb{Z})$ stabilizes with $n$ for each $k$. 
Barratt and Priddy, using the Quillen plus construction \cite{barrattpriddyone}, later identified the stable part of the homology with that of the zeroth component of a certain infinite loop space, that is, a topological space which has a grouplike $E_\infty$--algebra structure. 
They moreover explicitly determined this infinite loop space, which is the one underlying the sphere spectrum. 
\begin{teo}[Barratt--Priddy--Quillen \cite{barrattpriddytwo}]
There is a homology equivalence:
\[B\Sigma_\infty \overset{H_*}{\simeq} \Omega^\infty_0 S^\infty.\]
\end{teo}
Here $\Omega^\infty_0 S^\infty$ refers to the basepoint component of the infinite loop space of the sphere spectrum. 

This result was originally proved using Dyer--Lashof operations \cite{barrattpriddyone} following an argument of McDuff--Segal \cite{mcduffsegal} but was later reproven in many different ways -- some of which make use of scanning methods.
Segal first introduced the term of ``scanning'' in a study of homotopy types of spaces of holomorphic functions of certain closed surfaces \cite{segal79}, based on an idea of McDuff \cite{mcduff}.
It was later very successfully used to compute the stable homology of moduli spaces of manifolds \cite{madsenweiss, sorenrwModuliMonoids, gmtw} 
as well as graphs in order to compute the stable homology of automorphisms of free groups \cite{soren}. 

To get a better idea of the proof strategy, let us briefly examine the case of the space of embeddings of a $d$-dimensional manifold in $\mathbb{R}^N$, for large $N$.
There is a map, called a scanning map, from the space of embeddings of $M$ in $\mathbb{R}^N$ into an $N$-loop space $\Omega^N\Phi_N$. 
We view the $N$-loop space as the space of maps from the one-point compactification of $\mathbb{R}^N$ into the space $\Phi_N$ of local ``images''. 
Every embedding of the manifold $M$ into $\mathbb{R}^N$ determines an element in the loop space $\Omega^N\Phi_N$, via the map that associates to each point of $\mathbb{R}^N\cup\{\infty\}$ what one ``sees'' of the embedded manifold locally at that point. 
There are two possibilities: when the point is close to the manifold, it sees the tangent plane to the manifold, which is a $d$-dimensional plane, and when the point is far enough away from the manifold, it sees the empty space.  
Therefore, in this case, the space $\Phi_N$ turns out to be the tautological normal bundle associated to the compactification of the $d$-dimensional Grassmannian.
Letting $N$ tend to infinity results in an infinite loop space whose homology computes the stable homology of the moduli space of d-dimensional manifolds \cite{gmtw,grwI}.
Galatius, Madsen, Tillmann, and Weiss later reformulate this result in the context of cobordism categories \cite{gmtw}.

Galatius and Randal-Williams give a proof of the Barratt--Priddy--Quillen theorem \cite{sorenrwModuliMonoids} (Hatcher~\cite{hatcherMW}), which also uses scanning methods. 
This proof realizes the symmetric groups as a category of ``0-cobordisms'', that is, as a category where morphisms are manifolds of dimension 0, i.e. configurations of points. 
They then construct a scanning map from the space of cobordisms.

We give here a different proof, where an element in the symmetric groups is instead seen as a sort of 1-cobordism between configurations of points. 
We build a category internal to topological spaces where objects are configurations of points in $\mathbb{R}^\infty$ and morphisms are embedded paths between them. 
The classifying space of this category is equivalent to that of the symmetric groups. 
In order to construct a scanning map, we identify a space $\Phi_N$ of local images of the category, which is a sort of compactification of the morphism space.
Finally, we show that the spectrum constructed from the spaces $\Phi_N$ is equivalent to the sphere spectrum.

We develop this different proof because it generalizes better to other families of groups. 
In a follow-up paper \cite{delarue1}, we apply the methods used here to more complicated groups, the Higman--Thompson groups \cite{higman}. 
They are a generalization of the Thompson groups, which are groups of piecewise linear self-maps of a segment satisfying certain conditions.
Elements in the Higman--Thompson groups can be represented as trees, therefore as paths between configurations which are allowed to collide.

\paragraph{\textbf{Acknowledgements.}}
The author is grateful to Najib Idrissi and Nathalie Wahl for their encouragement. 
The author was partially supported by the project ANR-22-CE40-0008 SHoCoS, and hosted by the Danish National Research Foundation through the Copenhagen Center for Geometry and Topology (DNRF151).

\section{Preliminaries}

We start by recalling elementary notions about configuration spaces, as well as the language of semi-simplicial spaces, and microfibrations.
We let $I$ denote the unit interval $[0,1]$.
For $n\in\mathbb{N}$, let $\Sigma_n$ be the symmetric group on $n$ elements. Each group $\Sigma_n$ can be considered as a category, with one object denoted by $*_n$, and with space of endomorphisms the group $\Sigma_n$. Let us consider the category $\Sigma := \displaystyle \bigsqcup \Sigma_n$. 
It is equipped with a monoidal structure induced by the maps $\Sigma_n\times \Sigma_m\rightarrow \Sigma_{n+m}$.

\begin{defi}
The ordered configuration space of $n$ points in a topological space $X$ is the space
\[\Conf  (n,X):= \{(x_1,\dots, x_n)\in X^n \mid \forall i\neq j, x_i\neq x_j\}.\]
The symmetric group $\Sigma_n$ acts on this space by permuting the indices of the points in a configuration.
The unordered configuration space of $n$ points in $X$ is the quotient by this free action:
\[\UConf  (n,X):= \Conf  (n,X)/\Sigma_n.\]
\end{defi}

For every $n\in\mathbb{N}$, the ordered configuration space of $n$ points in $\mathbb{R}^\infty$ is weakly contractible \cite[Theorem~4.4]{itloopspacesmay}.
(A direct consequence of this is that the loop space at any basepoint of the ordered configuration space $\Omega_x \Conf (n, \mathbb{R}^\infty)$ is also contractible.)
The space $ \UConf (n, \mathbb{R}^\infty)$ therefore appears as the quotient of a contractible space by a free action of $ \displaystyle \Sigma_n$, so it is a model for $B\displaystyle \Sigma_n$.
 

\subsection{Simplicial and semi-simplicial spaces}

Let us briefly recall the basics of simplicial spaces. 
The model for $\Top$ used here is that of compactly generated Hausdorff spaces equipped with the Serre model structure.
Let $\Delta$ be the category with objects $[n] = \{0 < \dots < n\}$, and with morphisms from $[n]$ to $[m]$ the order-preserving set maps.
A simplicial space is then a functor $\Delta^{op}\rightarrow \Top$.
A cosimplicial space is a functor $\Delta \rightarrow \Top$.\\
We denote the topological $n$-simplex by $\Delta^n_{\Top}$:
\[\Delta^n_{\Top} := \{x\in\mathbb{R}^{n+1} \mid x_i \geq 0 \text{ for } 0\leq i\leq n, \text{ and } x_0+\dots +x_n = 1.\}\]
These form a cosimplicial space in a natural way, and play a role in defining the geometric realization of a simplicial space $\vert X_\bullet \vert$, 
which is the space $\bigsqcup_n \left( X_n\times \Delta^n_{\Top}/\sim\right)$ where $(x, f_*p)\sim (f^*x,p)$ for every simplicial map $f:[k]\rightarrow [l]$, $x\in X_l$, $p\in\Delta^k_{\Top}$.

To every category we can associate a simplicial set, its nerve, in a natural way.\\
The nerve of a category $\C$ is the simplicial set $N\C$ whose $n$-simplices are chains of $n$ composable morphisms:
\[(N\C)_n = \Mor(\C)\times_{\Ob(\C)} \dots \times_{\Ob(\C)} \Mor(\C).\]
The classifying space $B\C$ of a category $\C$ is the geometric realization of its nerve.

The categories that we work with have some extra structure, namely they are internal to topological spaces \cite{segal68}.
A category internal to topological spaces $\C$ is given by a pair of topological spaces called the object space $\Ob(\C)$ and the morphism space $\Mor(\C)$, 
together with continuous maps:
\begin{itemize}
    \item $i:\Ob(\C)\rightarrow \Mor(\C)$ which sends an object to the identity morphism;
    \item $s,t:\Mor(\C)\rightarrow \Ob(\C)$ which send an arrow to its source (resp. target);
    \item $\circ:\Mor(\C)\times_{t,s}\Mor(\C)\rightarrow \Mor(\C)$, which is composition of morphisms.
\end{itemize}
We can define the nerve of a topological category $\C$ in the same way as for regular categories, but we get a simplicial space instead of a simplicial set.

Let us now introduce semi-simplicial spaces, which we will often prefer to simplicial spaces.
Let $\Delta_+$ be the subcategory of $\Delta$ which contains all objects but only injective maps.
A semi-simplicial space is a functor $\Delta_+^{op}\rightarrow \Top$.
There is a forgetful functor from simplicial spaces to semi-simplicial spaces, induced by the inclusion $\Delta_+\hookrightarrow \Delta$. It amounts to forgetting the degeneracies. 
The \textit{thick} realization of a semi-simplicial space $X_\bullet$, $\vert\vert X_\bullet \vert\vert$, is defined as $ \bigsqcup_n \left(X_n\times \Delta^n_{\Top}/\sim\right)$, where $(x, f_*p)\sim (f^*x,p)$ for every face map $f:[k]\hookrightarrow [l]$ in $\Delta_+$.
For a simplicial space $X_\bullet$, the geometric realization is a quotient of the thick geometric realization.

Let $\C$ be a category. We can obtain two classifying spaces: the usual one, $B\C$, by taking the geometric realization of the nerve, and the thick one, $\mathbb{B}\C$, by taking the \emph{thick} realization of the nerve seen as a \emph{semi}-simplicial space.

For $X_\bullet$ a semi-simplicial space, the $n$th Segal map is the map
    \[X_n \rightarrow X_1\times_{X_0}\dots\times_{X_0} X_1\] induced by the maps $X_n\rightarrow X_1$ coming from the following maps $a^{i}$ in $\Delta_+$, for $0\leq i\leq n-1$:
    $$\begin{array}{rrr}
        a^{i}:[1] \rightarrow [n], & 0\mapsto i, & 1\mapsto i+1
    \end{array}$$
We say that a semi-simplicial space $X_\bullet$ is \emph{Segal} if its Segal maps are weak equivalences. For example, nerves of categories are Segal spaces.
A result of Segal \cite[Proposition~1.5]{segal74} states that for a simplicial Segal space with $X_1$ path-connected, there is a weak equivalence $X_1\simeq\Omega\vert\vert X_\bullet\vert\vert$.
The proof can be extended to the case of semi-simplicial spaces.

\begin{lem}\label{semiSegal}
Let $X_\bullet$ be a semi-simplicial Segal space where $X_1$ is path-connected. Then there is a weak equivalence $X_1\simeq\Omega\vert\vert X_\bullet\vert\vert$.
\end{lem}

The topological categories that we consider have various structures:
\begin{defi}\cite[Chapter 4]{may}
An $E_n$-algebra in a symmetric monoidal category $\C$ is an object $X$ of $\C$ equipped with maps $E_n(k)\otimes X^{\otimes k}\rightarrow X$.
The model we use for $E_n$ is that of the little $n$-cubes operad.
\end{defi}

\subsection{Microfibrations.}
Let us now introduce a specific type of map called a Serre microfibration. 
These are maps that satisfy a weaker version of the homotopy lifting property than normal fibrations. 
However, they still have good properties.
\begin{defi}
A map $p:E\rightarrow B$ is called a Serre microfibration if, for every $k\geq 0$ and every commutative square:
\begin{center}
    \begin{tikzcd}
    D^k\times\{0\}\arrow{d} \arrow{r} &E\arrow{d}{p}\\
    D^k\times[0,1]\arrow{r} &B,
    \end{tikzcd}
\end{center}
there exists $\epsilon>0$ and a lifting $h:D^k\times [0,\epsilon]\rightarrow E$ such that the following diagram commutes:
\begin{center}
\begin{tikzcd}[row sep = 0.4cm, column sep = 0.4cm]
D^k\times\{0\} \arrow[rr] \arrow[dd] \arrow[dr] && E\arrow[dd]\\
&D^k\times[0,\epsilon] \arrow[dl]\arrow[ur,dashrightarrow]\\
D^k\times[0,1]\arrow[rr] &&B.
\end{tikzcd}
\end{center}
\end{defi}

\begin{exemple}
    A Serre fibration is a Serre micro-fibration.
    Let $p:E\rightarrow B$ be a Serre fibration, and $V$ an open subset of $E$. Then the restriction of $p$ to $V$ is a Serre microfibration, but not necessarily a Serre fibration (see for example Raptis in \cite{raptis}.)
\end{exemple}

\begin{lem}\cite[Lemma~2.2]{weisslemma}
A Serre micro-fibration with weakly contractible fibers is a Serre fibration, and therefore a weak homotopy equivalence.
\end{lem}

The following proposition will enable us to compare the usual and thick realizations.
\begin{prop}\label{levelwisereal}\cite[Propositions~2.7-2.8]{grwI}.
\begin{itemize}
\item[$\bullet$] Let $f_\bullet:X_\bullet\rightarrow Y_\bullet$ be a map of semi-simplicial spaces, and $n\in\mathbb{N}$ such that for all $p\in\mathbb{N}$, $f_p:X_p\rightarrow Y_p$ is $(n-p)$-connected. Then $||f_\bullet||:||X_\bullet||\rightarrow ||Y_\bullet||$ is $n$-connected.
\item[$\bullet$] Consider a semi-simplicial set $Y_\bullet$ (viewed as a discrete space) and a Hausdorff space $Z$. Let $X_\bullet \subset Y_\bullet \times Z$ be a sub-semi-simplicial space which is an open subset in every degree.
Then the map $\pi:||X_\bullet|| \rightarrow Z$ is a Serre microfibration.
\end{itemize}
\end{prop}
A statement analogous to the first point for the regular geometric realization is false unless the simplicial space satisfies a certain condition, mentioned in definition \ref{proper}.

\begin{defi}\cite[Definition~A.4]{segal74}\label{proper}
A simplicial space $X_\bullet$ is \textit{good} if the maps $s_i(X_{p-1})\rightarrow X_p$ are closed Hurewicz cofibrations for every $i$ and $p$.
\end{defi}

\begin{lem}\label{realizations}\cite[Proposition~A.1.(iv)]{segal74}
Let $X$ be a simplicial space. If $X_\bullet$ is good, then the quotient map $||X_\bullet||\rightarrow |X_\bullet|$ is a weak equivalence.
\end{lem}

Because the semi-simplicial spaces we consider are often the nerve of a category, we'd like to introduce a criterion that determines whether the semi-simplicial and simplicial nerves of a category are equivalent.
We will say that a topological category $\C$ is \emph{well-pointed} if its nerve $N\C$ is a good simplicial space. 

\begin{defi}\label{NDRgood}
For $B$ a topological space, let $\mathrm{CGH}/B$ be the category of compactly generated Hausdorff spaces over $B$. 
A neighbourhood-deformation-retract pair (or NDR-pair) over $B$ is a pair $(X,A)$ in $\mathrm{CGH}/B$, with $A\subset X$, such that there exist maps $u:X\rightarrow I\times B$ and $h: X\times I\rightarrow X$ over $B$, i.e. commutative diagrams:
\begin{center}
\begin{minipage}{0.4\linewidth}
\begin{center}
    \begin{tikzcd}
    X\arrow[dr]\arrow[rr,"u"] && I\times B \arrow[dl,"pr_2"]\\
    &B
    \end{tikzcd}
\end{center}
\end{minipage}
\begin{minipage}{0.4\linewidth}
\begin{center}
    \begin{tikzcd}
    X\times I\arrow[dr]\arrow[rr,"h"] && X \arrow[dl]\\
    &B.
    \end{tikzcd}
\end{center}
\end{minipage}
\end{center}
satisfying the following conditions:
\begin{enumerate}
    \item $A = u^{-1}(\{0\}\times B)$,
    \item $h(-,0) = \id_X$,  $h\vert_{A\times I} = pr_A$,
    \item $h(x,1)\in A$ for $x\in u^{-1}([0,1)\times B)$.
\end{enumerate}
Note that the map $X\times I\rightarrow B$ does not have to be the composition of the projection on $X$ and then the map $X\rightarrow B$.
\end{defi}

\begin{prop}\cite[Proposition~10]{TopQuillenA}
Let $\C$ be a topological category. If $(\Mor(\C), \Ob(\C))$ is an NDR-pair over $\Ob(\C)\times \Ob(\C)$, then $N\C$ is a good simplicial space.
\end{prop}

\section{A topological category}

Let us define a topological category which serves as a model for the classifying space of $\displaystyle \sqcup \Sigma_n$. 
Objects are configurations of points in $I^N$ for large $N$ and morphisms between configurations of the same number of points are paths in $I^N\times \mathbb{R}$ which start at one configuration, end at another, and never intersect each other.\\
In the product $I^N\times\mathbb{R}$, the second factor $\mathbb{R}$ embodies a ``time'' dimension which is always increasing.

\begin{defi}
Let $N$ be in $\mathbb{N}\cup\{\infty\}, N\geq 3$. 
Let $\C_N$ be the topological category where:
\begin{itemize}
    \item the space of objects is $\displaystyle \bigsqcup_n \UConf (n, I^N)\times\mathbb{R}$.
    \item the space of morphisms is the space of all $(t_x,t_y,\phi)$
    where $t_x,t_y\in\mathbb{R}$, $\phi:[t_x,t_y] \rightarrow \UConf(n,I^N)\times\mathbb{R}$ for some $n$ and:
    $$\forall t\in [t_x,t_y], \phi(t) = (z_\phi(t), t),$$
    for some continuous map $z:\mathbb{R}\rightarrow \UConf(n,I^N)$.
    
\end{itemize}
The source and target maps $s,t: \Mor(\C_N)\rightarrow \Ob(\C_N)$ are given by:
\begin{align*}
s(t_x, t_y,\phi) &= \phi(t_x)\\
t(t_x, t_y, \phi) &= \phi(t_y).
\end{align*}
The space of morphisms from $(x,t_x)$ to $(y,t_y)$ for $\vert x\vert \neq \vert y\vert$ or $t_x>t_y$ is empty.
Composition of $(\phi, t_x, t_y)$ and $(\psi,t_y,t_z)$ is given by the concatenation of $\phi$ and $\psi$ denoted by $(\phi,\psi)$. More explicitly:
\begin{align*}
(\phi, \psi): [t_x,t_z] &\rightarrow I^N\times \mathbb{R}\\
t &\mapsto \phi(t) \text{ if } t_x\leq t\leq t_y \\
t&\mapsto \psi(t) \text{ if } t_y\leq t\leq t_z.
\end{align*}
\end{defi}

\begin{rema}
    The category $\C_N$ is the coproduct $\displaystyle \sqcup_{n} \C_N(n)$, where $\C_N(n)$ is the full subcategory of $\C_N$ obtained by restricting the space of objects to $\UConf(n,I^N)$.
\end{rema}

\begin{rema}
    The category $\C_N$ has an $E_N$-structure which induces an $E_N$-algebra structure on its classifying space.
\end{rema}

We have introduced this category because we are interested in its classifying space. 
Since semi-simplicial sets are easier to work with, let us check that $\C_N$ behaves well with respect to realizations, i.e. that it is well-pointed.

\begin{lem}\label{goodnerve}
Let $N\in\mathbb{N}$ or $N=\infty$. The nerve of the category $\C_N$ is good (i.e., the category $\C_N$ is well-pointed).
\end{lem}

\begin{proof}
We want to define maps $u$ and $h$ satisfying the conditions of definition \ref{NDRgood}.\\
The space of objects of $\C_N$, denoted $\Ob(\C_N)$, is identified inside $\Mor(\C_N)$ as the space of identities $\Id(\C_N)$. 
Note that the map $\Mor(\C_N)\rightarrow \Ob(\C_N)\times\Ob(\C_N)$ is the source-target map, 
while the map $\Mor(\C_N)\times I\rightarrow \Ob(\C_N)\times\Ob(\C_N)$ sends $(\phi,t_x,t_y,s)$ to $(\phi(t_x),\phi(st_x+(1-s)t_y))$.
Let 
\begin{align*}
h: \Mor(\C_N)\times I &\longrightarrow \Mor(\C_N)
\end{align*}
be a homotopy which progressively truncates a path until it has length 0.
More precisely, it takes a tuple $(\phi, t_x, t_y, s)$ to the tuple $(\phi\vert_{[t_x; st_x+(1-s)t_y]},t_x,st_x+(1-s)t_y)$.
We note that $h(\phi, t_x,t_y, 0) = (\phi,t_x,t_y)$, $h(\id_{(x,t_x)},s) = \id_{(x,t_x)}$, and $h(\phi, t_x, t_y, 1) = \id_{\phi(t_x)} \in\Id(\C_N)$.

Let us set
\begin{align*}
    u: \Mor_{\id}(\C_N) &\longrightarrow I\times \Ob(\C_N)\times\Ob(\C_N)\\
    (\phi,t_x,t_y) &\longmapsto \left(\min(t_y-t_x,1),\phi(t_x), \phi(t_y)\right)
\end{align*}

These two maps make both diagrams in definition \ref{NDRgood} commute. 
Moreover,
\[u^{-1}(\{0\}\times \Ob(\C_N)\times \Ob(\C_N)) = \Id(\C_N).\]
Therefore, the category $\C_N$ is well-pointed.
\end{proof}

This category is a topological model for the symmetric groups: 
\begin{prop}\label{homeq}
    For every $N$, there is an $(N-1)$-connected map of $E_N$-algebras
    \[B\C_N\rightarrow\bigsqcup B\Sigma_n.\]
\end{prop}

Before we move on to the proof of proposition \ref{homeq}, let us briefly introduce the space of Moore paths of a topological space.
\begin{defi}
Let $X$ be a topological space. Let \[P^{\Moore}X := \{(f,r)\vert f:[0,r]\rightarrow X, r\geq 0\}.\] 
This space can be viewed as a subspace of $\Top([0,\infty),X)\times\mathbb{R}$ by setting $f(u) = f(r)$ for $u\geq r$.
\end{defi}
\begin{lem}\label{Moorepathseq}
The space of Moore paths $P^{\Moore}X$ is homotopy equivalent to the space of paths $PX = \{\gamma:I\rightarrow X\}$.
\end{lem}
\begin{proof}
Let us consider the inclusion $i:PX \hookrightarrow P^{\Moore}X$.
Let us define the map \begin{align*}
h:P^{\Moore}X \rightarrow PX\end{align*} which renormalizes a Moore path of length $r>1$ to a path of length 1, and extends a path of length $r<1$ by $f(u)=f(r)$ for $u\in[r,1)$.
This map is such that $h\circ i = \id$ and $i\circ h \simeq \id$.
\end{proof}
\begin{prop}
For $X$ a path connected space, the path space $PX$ is homotopy equivalent to $X$.\qed
\end{prop}

\begin{proof}[Proof of proposition \ref{homeq}.]
First, note that the space of objects $N_0\C_N$ is equivalent to $\bigsqcup_n \UConf(n,\mathbb{R}^N)$. 
The space of morphisms is equivalent to the Moore path space of $\bigsqcup \UConf(n,\mathbb{R}^N)$ via the following maps:
\begin{align*}
\Mor(\C_N) &\rightarrow P^{\Moore}\left(\bigsqcup \UConf(n,I^N)\right)\\
(\phi,t_x,t_y) &\mapsto (z_\phi, t_y-t_x)
\end{align*}
and
\begin{align*}
    P^{\Moore}\left(\bigsqcup \UConf(n,I^N)\right) &\rightarrow \Mor(\C_N)\\
    (\phi,r) &\mapsto ((\phi(t),t),0,r).
\end{align*}

We can extend these maps to a simplicial map from $N_\bullet \C_N$ into the singular simplicial set $\Sing_\bullet( \sqcup \UConf(n,I^N)) := \map(\Delta^\bullet_{\Top}, \displaystyle \sqcup \UConf(n,I^N))$.
As stated earlier, this map is an equivalence in levels 0 and 1. Because both $N_p\C_N$ and the singular simplicial space are Segal spaces, this equivalence extends to all levels.
Therefore, the simplicial map is a levelwise equivalence, and by lemma \ref{levelwisereal}, it induces an equivalence on thick realizations.

To compare this with the regular realization, we need to prove that $\Sing_\bullet(X)$ is a good simplicial space for $X$ a metric space.
The degeneracies $s_i:\Sing_n(X)\rightarrow \Sing_{n+1}(X)$ are induced by the cosimplicial maps $\sigma^i:\Delta^{n+1}_{\Top}\twoheadrightarrow \Delta^n_{\Top}$.
We can build an explicit neighbourhood deformation retraction of $\map(\Delta^{n+1}_{\Top},X)$ onto $\map(\Delta^{n}_{\Top},X)$ via the maps:
\begin{align*}
h: \map(\Delta^{n+1}_{\Top},X)\times I&\rightarrow \map(\Delta^{n+1}_{\Top},X)\\
(f,s)&\mapsto f(x_0,\dots, x_i+sx_{i+1}, (1-s)x_{i+1}, \dots, x_{n})\\
u:\map(\Delta^{n+1}_{\Top},X)&\rightarrow I\\
f&\mapsto \sup_{\underline{t}\in\Delta^n_{\Top}}\sup_{\underline{s}\in(\sigma^i)^{-1}(t)} d(f(\underline{s}),f(\underline{t})).
\end{align*}
Therefore, the maps $s_i(\Sing_n(X))\hookrightarrow \Sing_{n+1}(X)$ are cofibrations and $\Sing_\bullet(X)$ is a good simplicial space.
There is therefore an equivalence: \[B\C_N\simeq \vert \Sing_\bullet(\sqcup \UConf(n,\mathbb{R}^N))\vert \simeq \sqcup \UConf(n,\mathbb{R}^N).\]
Now, the inclusion $\UConf(n,\mathbb{R}^N)\hookrightarrow \UConf(n,\mathbb{R}^\infty)$ is $(N-1)$-connected, and $\UConf(n,\mathbb{R}^\infty)$ is a model for $B\Sigma_n$ as previously stated, so the proposition follows.
\end{proof}

\section{Delooping result}
In this section, we will ``zoom in'' on the classifying space of our category in order to find an expression of it as a loop space.

\subsection{Resolution of the classifying space of the category $\C_\infty$}

\begin{defi}\label{weirdtop}
    Fix some $n,N\in\mathbb{N}$, $0\leq k \leq N$, and $J$ a subinterval of $\mathbb{R}$. Let $\Psi_n(J\times \mathbb{R}^{k}\times I^{N-k})$ be the space of all $\phi : J \rightarrow \mathbb{R}\times \UConf(n,\mathbb{R}^{k}\times I^{N-k})$ satisfying the following condition:
    $$\forall t\in J,\ \phi(t) = (t, z_\phi(t)) \text{ for some } z_\phi:J \rightarrow \UConf(n,\mathbb{R}^{k}\times I^{N-k}).$$
    \end{defi}

  For $\phi\in \Psi_n(J\times \mathbb{R}^{k}\times I^{N-k})$ and $P$ any submanifold of $J\times\mathbb{R}^{k}\times I^{N-k}$ of the form $J'\times P'$ where $J'\subset J$ and $P'$ is a submanifold of $\mathbb{R}^{k}\times I^{N-k}$, we denote by $\phi\cap P$ the intersection $\im(\phi)\cap P$.
    \begin{defi}
    Let $P$ be a submanifold of $J\times\mathbb{R}^{k}\times I^{N-k}$ of the form $J'\times P'$, where $J'$ is a subinterval of $J$ and $P'$ a submanifold of $\mathbb{R}^{k}\times I^{N-k}$. 
    We say that $\phi\sim_P \psi$ if $\phi\cap P = \psi\cap P$.
    Let 
    \[\Phi_k^N(J\times\mathbb{R}^{k}\times I^{N-k},P) := \bigg( \displaystyle\bigsqcup_n \Psi_{n}(J\times\mathbb{R}^{k}\times I^{N-k})\bigg)/ \sim_P.\]
   We will apply this to the case where $J'_R = J\cap [-\frac R2;\frac R2]$ and $P'$ is the cube $B_R^{k}(0)$ of center 0 and side $R$ in dimension $k$.
    For $0\leq k\leq N$, let us set $$\Phi_{k}^{N}(J) := \lim_{R\in\mathbb{N}} \Phi_k^N(J\times\mathbb{R}^{k}\times I^{N-k}, J'_R\times B_R^{k}(0) \times I^{N-k}).$$
    When $J=\mathbb{R}$, we omit it from the notation.
    \end{defi}
    
    \begin{figure}[h!]
    \tiny
    \def\svgwidth{\hsize}
    \begingroup%
    \makeatletter%
    \providecommand\color[2][]{%
      \errmessage{(Inkscape) Color is used for the text in Inkscape, but the package 'color.sty' is not loaded}%
      \renewcommand\color[2][]{}%
    }%
    \providecommand\transparent[1]{%
      \errmessage{(Inkscape) Transparency is used (non-zero) for the text in Inkscape, but the package 'transparent.sty' is not loaded}%
      \renewcommand\transparent[1]{}%
    }%
    \providecommand\rotatebox[2]{#2}%
    \newcommand*\fsize{\dimexpr\f@size pt\relax}%
    \newcommand*\lineheight[1]{\fontsize{\fsize}{#1\fsize}\selectfont}%
    \ifx\svgwidth\undefined%
      \setlength{\unitlength}{496.06299213bp}%
      \ifx\svgscale\undefined%
        \relax%
      \else%
        \setlength{\unitlength}{\unitlength * \real{\svgscale}}%
      \fi%
    \else%
      \setlength{\unitlength}{\svgwidth}%
    \fi%
    \global\let\svgwidth\undefined%
    \global\let\svgscale\undefined%
    \makeatother%
    \begin{picture}(1,0.55428571)%
      \lineheight{1}%
      \setlength\tabcolsep{0pt}%
      \put(0,0){\includegraphics[width=\unitlength,page=1]{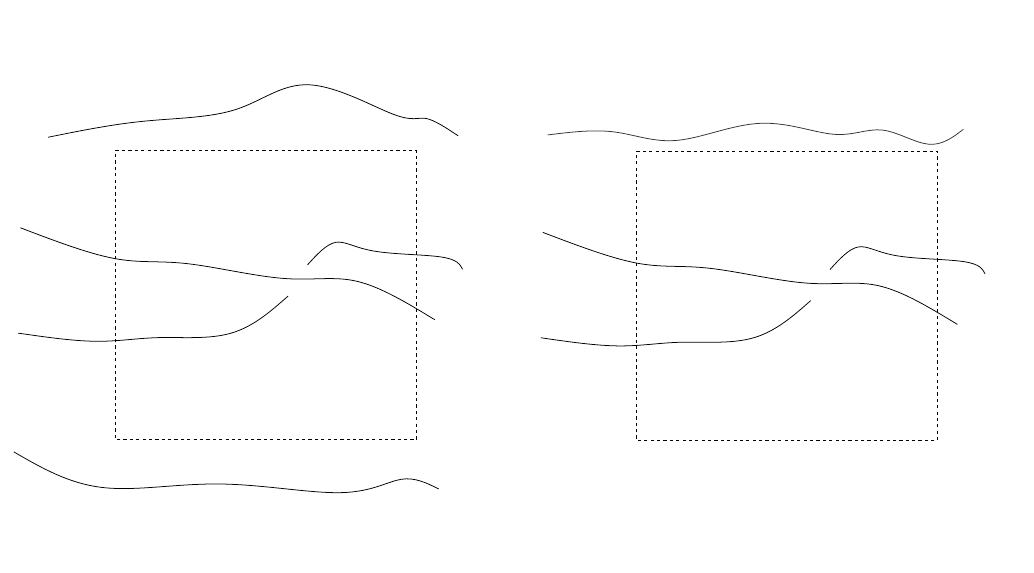}}%
    \end{picture}%
  \endgroup%
  
    \caption{Two elements in $\Phi_1^1$ that are connected by a path.}
    \label{twoelements}
    \end{figure}
    We can think of $\Phi_0^N$ as a space of long paths of configurations that we only need to understand on every compact interval.

    For each $k\geq 0$, there is a point at infinity in $\Phi^{N}_k$, which corresponds to the empty path. It is the basepoint of this space.
    \begin{prop}\label{pathconnected}
        For $k\geq 1$, the space $\Phi_{k}^{N}$ is path-connected.
    \end{prop}
    
    \begin{proof}
    When $k\geq 1$, we build a path between any element in $\Phi_{k}^N$ and the point at infinity by progressively sliding it away along the first dimension (or any dimension which is allowed to go to infinity, i.e. between $1$ and $k-1$).
    \end{proof}

Let us now give a technical lemma which will be useful for the proof of proposition \ref{deloopinglevelone}.

\begin{lem}\label{posetthing}
    Let $P$ be a subposet of $\mathbb{R}$.
    The classifying space of the topological poset $P$ is contractible.
\end{lem}

\begin{proof}
Let us consider the discrete version of this poset, $P^\delta$. The identity map $i: P^\delta\hookrightarrow P$ is continuous.\\ 
Let $D_{p,q} = N_{p+q+1}P$, viewed as a topological subspace of $N_pP\times N_qP^\delta$ via the map which takes a $(p+q+1)$-simplex in $NP$ and restricts it to its first $p$ coordinates on $NP$ and last $q$ coordinates on $NP^\delta$.
This yields a semi-bisimplicial space $D$, with maps:
\begin{align*}
    D_{p,q} &\xrightarrow{\alpha_{p,q}} N_pP\\
    D_{p,q} &\xrightarrow{\beta_{p,q}} N_qP^\delta.
\end{align*}

By proposition \ref{levelwisereal}, the map $\alpha_p: \vert\vert D_{p,\bullet}\vert\vert\rightarrow N_pP$ is a Serre microfibration. 
The fiber of $\alpha_p$ over $(t_0,\dots, t_p) \in N_pP$ is the discrete space \[\{q\in\mathbb{N}, (s_0,\dots, s_q)\in P^\delta\backslash\{t_0,\dots, t_p\} \}.\]
This is the nerve of the discrete ordered set $P^\delta\backslash\{t_0,\dots, t_p\}$, which is contractible. 
Therefore the map $\alpha_p$ is a homotopy equivalence. 
By proposition \ref{levelwisereal}, the map \[\vert\vert\alpha\vert\vert: ||D_{\bullet,\bullet}||\rightarrow ||N_\bullet P||\] is a homotopy equivalence. 
But $\vert\vert i\vert\vert \circ ||\beta||\simeq ||\alpha||$ \cite[Lemma~5.8]{grwI}, so $\vert\vert\alpha\vert\vert$ factors up to homotopy through the space $|\vert N_\bullet P^\delta|\vert$ which is contractible since it is the realization of the nerve of a discrete directed poset. 
Thus, both spaces $||D_{\bullet,\bullet}||$ and $||N_\bullet P||$ are contractible.
It now remains to show that $P$ is well-pointed as a topological poset.
We set 
\begin{align*}
    u:\Mor(P)\rightarrow I\times\Ob(P)\times\Ob(P) &&h:\Mor(P)\times I&\rightarrow \Mor(P)\\
    f_{x,y} \mapsto (\min(1,y-x),x,y) &&(f_{x,y},t)&\mapsto f_{x,tx+(1-t)y} 
\end{align*}
Therefore, the thick geometric realization of $P$ is equivalent to its regular geometric realization and the space $|N_\bullet P|$ is contractible. 
\end{proof}


\begin{prop}\label{deloopinglevelone}
For $N\geq 1$, there is a zigzag of weak equivalences of $E_N$-algebras:
\[\Phi_0^N \simeq B\C_N.\]
\end{prop}

\begin{proof}
Let us construct a semi-simplicial resolution $X_\bullet$ of $\Phi_0^N.$
The space of $p$-simplices $X_p$ is given by $\Phi_0^N\times N_p\mathbb{R}$, where $\mathbb{R}$ is seen as a topological poset.
The face maps are given by, for $0\leq k \leq p$:
\begin{align*}
    d_k: X_p&\rightarrow X_{p-1}\\
    (\phi,t_0,\dots, t_p) &\mapsto (\phi,t_0,\dots, \hat{t_k},\dots, t_p).
\end{align*}
There is a semisimplicial map $\eta: X_\bullet\rightarrow N_\bullet\C_N$, which is defined on $p$-simplices as the map which sends $(\phi,t_0,\dots, t_p)$ to $\left((\phi_{\vert_{[t_0,t_1]}},t_0,t_1),\dots,(\phi_{\vert_{[t_{p-1},t_p]}},t_{p-1},t_p)\right)$. 
Here, $\phi_{\vert_{[t_{i-1},t_i]}}$ is an element of $\Phi_0^N([t_{i-1};t_i])$.
(Restricting $\phi$ to compact subintervals of $\mathbb{R}$ is continous with respect to the limit topology).
Informally, the map $\eta$ forgets what happens before $t_0$ and after $t_p$ and restricts $\phi$ to $p$ subintervals in between.

The projection $X_p\rightarrow N_p\mathbb{R}$ sending $(\phi, t_0,\dots, t_p)$ to $(t_0,\dots, t_p)$ is a fibration. 
Since the target of the fibration is contractible, the space $X_p$ is equivalent to the fiber over a fixed point $(t_0,\dots, t_p)$.
Given a $p$-tuple $(t_0,\dots,t_p)$, an element $\phi$ in the fiber is entirely determined by its restriction to all the intervals $(-\infty,t_0],\dots,[t_p,\infty)$ provided their endpoints coincide pairwise.
This means that the fiber $X'_p$ of the projection can be expressed as:
\begin{align}
X'_p =\Phi_0^N((-\infty,t_0])\times_{\UConf(n,I^N)}\dots\times_{\UConf(n,I^N)}&\Phi_0^N([t_k,t_{k+1}]) \label{fiberprime}\\
\times_{\UConf(n,I^N)}\dots\times_{\UConf(n,I^N)}\Phi_0^N([t_p,\infty)) \nonumber.
\end{align}

The map $N_p\C_N\rightarrow N_p\mathbb{R}$ is also a fibration, so let us consider the following diagram:
\begin{center}
    \begin{tikzcd}
        X'_p \arrow{r}\arrow{d}&Y'_p\arrow{d}\\
        X_p\arrow{r}{\eta_p}\arrow{d} &N_p\C_N\arrow{d}\\
        N_p\mathbb{R}\arrow{r}{=}&N_p\mathbb{R}.
    \end{tikzcd}
\end{center}
The restriction of $\eta_p$ to $X'_p$ is a map into the space
\[Y'_p:= \Phi_0^N([t_0,t_1])\times_{\UConf(n,I^N)}\dots\times_{\UConf(n,I^N)}\Phi_0^N([t_{p-1},t_{p}])\]
which forgets the two outer terms of the product in expression \ref{fiberprime}. 
We therefore need to prove that this operation does not change the homotopy type of the product.

The target maps from each $\Phi\left((-\infty,t_0]\times I^N, ([-R,R]\cap(-\infty,t_0])\times I^N\right)$ -- resp. source maps from $\Phi\left([t_p,\infty)\times I^N, ([-R,R]\cap[t_p,\infty))\times I^N\right)$ -- induce maps from the limit:
\begin{align*}
p_0: \Phi_0^N((-\infty,t_0])&\rightarrow \UConf(n,I^N)\\
p_p: \Phi_0^N([t_p,\infty))&\rightarrow \UConf(n,I^N)
\end{align*}
which are also fibrations, so the two outer pullbacks in (\ref{fiberprime}) are homotopy pullbacks.

Let us show that the two outer terms in this (homotopy) pullback are equivalent to $\UConf(n,I^N)$
and so that removing them does not change the homotopy type.
Let us consider the map $q_0:\UConf(n,I^N)\rightarrow \Phi_0^N((-\infty,t_0])$ induced by the maps:
$q_R:\UConf(n,I^N)\rightarrow \Psi((-\infty,t_0]\times I^N, [-R,R]\times I^N)$ which send a configuration $x$ to the constant path $\phi:t\mapsto (x,t)$.
Clearly, $p_0\circ q_0 = \id$. The following is a homotopy between $q_0\circ p_0$ and $\id$:
\begin{align*}
H: \Phi_0^N((-\infty,t_0])\times I &\rightarrow \Phi_0^N((-\infty,t_0])\\
(\phi,t_0,s) &\mapsto (\gamma_s,t_0),
\end{align*}
where $\gamma_1$ is the constant path at $\phi(t_0)$ and for $s<1$:
$$\gamma_s:t \mapsto \begin{cases}
&\phi(t-t_0-\frac{s}{1-s}),\text{ if }t\in(-\infty,t_0-\frac{s}{1-s}];\\ 
&\phi(t_0),\text{ if }t\in[t_0-\frac{s}{1-s},t_0].
\end{cases}$$ 
Therefore, the map $p_0$ is an equivalence.
Similarly, the map $p_p$ is an equivalence.

This proves that the lift of $\eta_p$ to fibers is an equivalence.
The map $\eta_p$ is a homotopy equivalence for all $p$, and so the map $\eta$ induces an equivalence on weak realizations:
\[\vert\vert X_\bullet\vert\vert \simeq \vert\vert N_\bullet\C_N\vert\vert,\]
but the thick realization of the nerve is equivalent to the ordinary geometric realization by lemma \ref{goodnerve}, therefore:
\[\vert\vert X_\bullet\vert\vert \simeq B\C_N.\]

Because $\vert\vert X_\bullet\vert\vert = \Phi_0^N\times \mathbb{B}\mathbb{R}$, there is a projection map 
\begin{align}
\epsilon: \vert\vert X_\bullet\vert\vert\rightarrow \Phi_0^N 
\end{align}
which forgets the $(t_k)$.
This map is a fibration.

The fiber over a point $\phi\in\Phi_0^N$ is the classifying space of the topological poset $\mathbb{R}$, which we show to be contractible in lemma \ref{posetthing}.
The map $\epsilon$ has contractible fibers so it is a homotopy equivalence, yielding the following zigzag:
\[\Phi_0^N\xleftarrow{\simeq}\vert\vert X_\bullet\vert\vert \xrightarrow{\simeq} B\C_N.\]
Additionally, these maps are compatible with the $E_N$-structures on $\Phi_0^N$ and $\C_N$.
In particular, when $N=\infty$, we get:
\[\Phi_0^\infty\xleftarrow{\simeq}\vert\vert X_\bullet\vert\vert \xrightarrow{\simeq} B\C_\infty. \qedhere\]

\end{proof}

\subsection{Zooming in, in higher dimensions}
Now, we will ``zoom in'' dimension by dimension in order to map into a space of local images of paths of configurations.

Informally, the idea of scanning is to construct a map $\mathbb{R}^N\times \Phi_0^N\rightarrow \Phi_N^N$, which sends a pair $(x,\phi)$ to the part of $\phi$ contained in a small neighborhood of the point $x$ as if there were a magnifying glass set at the point $x$ (see figure \ref{scanning}).
\begin{figure}[h!]
    \tiny
    \def\svgwidth{\hsize}
    \begingroup%
    \makeatletter%
    \providecommand\color[2][]{%
      \errmessage{(Inkscape) Color is used for the text in Inkscape, but the package 'color.sty' is not loaded}%
      \renewcommand\color[2][]{}%
    }%
    \providecommand\transparent[1]{%
      \errmessage{(Inkscape) Transparency is used (non-zero) for the text in Inkscape, but the package 'transparent.sty' is not loaded}%
      \renewcommand\transparent[1]{}%
    }%
    \providecommand\rotatebox[2]{#2}%
    \newcommand*\fsize{\dimexpr\f@size pt\relax}%
    \newcommand*\lineheight[1]{\fontsize{\fsize}{#1\fsize}\selectfont}%
    \ifx\svgwidth\undefined%
      \setlength{\unitlength}{287.23888968bp}%
      \ifx\svgscale\undefined%
        \relax%
      \else%
        \setlength{\unitlength}{\unitlength * \real{\svgscale}}%
      \fi%
    \else%
      \setlength{\unitlength}{\svgwidth}%
    \fi%
    \global\let\svgwidth\undefined%
    \global\let\svgscale\undefined%
    \makeatother%
    \begin{picture}(1,0.2292096)%
      \lineheight{1}%
      \setlength\tabcolsep{0pt}%
      \put(0,0){\includegraphics[width=\unitlength,page=1]{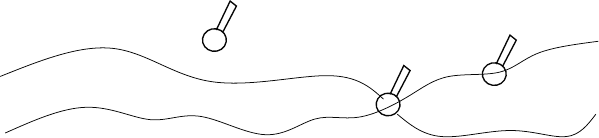}}%
    \end{picture}%
  \endgroup%
    \caption{Idea of the scanning map}
    \label{scanning}
    \end{figure}
When $x$ goes to infinity, $(\phi,x)$ is sent to the basepoint, so the map above factors through $S^N$. 
This yields a map $s:\Phi_0^N\rightarrow \Omega^N\Phi_N^N$, called the scanning map.
In practice, showing that this map is an equivalence is more technical and requires zooming in on each dimension separately. 

Let us compute the classifying space of $\Phi_0^N$:
\begin{prop}\label{deloopingleveltwo}
    There is an equivalence $B \Phi_0^N \simeq \Phi_1^N$.
\end{prop}

\begin{proof}
Let us strictify $\Phi_0^N$ into a monoid $\tilde{\Phi}_0^N$ where elements of $\tilde{\Phi}_0^N$ are tuples $(\phi,a,b)$, where $a,b\in\mathbb{R}$ and $\phi$ is the image by an affine map 
of an element in $\Phi_0^N$ so that it is a map from $\mathbb{R}$ to $\mathbb{R}\times [a,b]\times I^{N-1}$.
The monoid operation is obtained by stacking.
There is an equivalence $\tilde{\Phi}_0^N\simeq \Phi_0^N$.

Let us show that $B \tilde{\Phi}_0^N \simeq \Phi_1^N$.

We use the same method as in the proof of proposition \ref{deloopinglevelone}, so we will go into less detail in this proof.
Let us build a semi-simplicial space $Z_\bullet$ which will fit into a zigzag of the form
\[B \tilde{\Phi}_0^N \leftarrow \vert\vert Z_\bullet\vert\vert \rightarrow \Phi_1^N.\]
Let $Z_p$ be the subset of $\Phi_{1}^N\times N_p\mathbb{R}$ containing all $(\phi,t_0,\dots, t_p)$ such that $\phi$ does not intersect the hyperplanes $\mathbb{R}^{1}\times \{t_k\}\times I^{N-1}$, for every $0\leq k \leq p$.
The $i$-th face map forgets $t_i$.

There is a semisimplicial map $\eta: Z_\bullet\rightarrow B_\bullet \tilde{\Phi}_0^N$:
on $p$-simplices, this map sends $(\phi,t_0,\dots, t_p)$ to the chain $\left((\phi_1,t_0,t_1),\dots, (\phi_p,t_{p-1},t_p)\right)$,
where $\phi_1$ is the intersection of $\phi$ with $\mathbb{R}\times [t_0,t_p]\times I^{N-1}$.
Informally, the map $\eta_p$ forgets what happens before $t_0$ and after $t_p$.

There is a clear inclusion $i_p: B_p \tilde{\Phi}_0^N \rightarrow Z_p$.
The composition $\eta_p\circ i_p$ is the identity, and the composition $i_p\circ \eta_p$ is homotopic to the identity via the homotopy 
which slides further and further away the components of $\phi$ before $t_0$ and after $t_p$ to the point at infinity of $Z_p$.
The map $\eta_p$ is a homotopy equivalence for all $p$, and so the map $\eta$ induces an equivalence on weak realizations:
\[\vert\vert Z_\bullet\vert\vert \simeq  \vert\vert B_\bullet \tilde{\Phi}_0^N\vert\vert\]

The thick realization of the monoid is equivalent to the ordinary geometric realization, so:
\[\vert\vert Z_\bullet\vert\vert \simeq B\tilde{\Phi}_0^N.\]

There is a projection map $\epsilon: \vert\vert Z_\bullet\vert\vert \rightarrow \Phi_1^N$.
which at each level sends $(\phi,t_0,\dots, t_p)$ to $\phi$. 
Let us show that this map is a microfibration by considering the square:
\begin{center}
    \begin{tikzcd}
    D^k\times\{0\} \arrow[rr,"A"] \arrow[dd] \arrow[dr] && \vert\vert Z_\bullet\vert\vert\arrow[dd,"\epsilon"]\\
    &D^k\times[0,\delta] \arrow[dl]\arrow[ur,dashrightarrow,"H"]\\
    D^k\times[0,1]\arrow[rr,"B"] &&\Phi_1^N,
    \end{tikzcd}
\end{center}
where $A(d) = (\phi(d),t_0(d),\dots, t_p(d),x(d))$ and $B(d,s) = (\psi(d,s))$,
we can find a lift $H$ by setting $H(s) = (\psi(d,s),t_0(d),\dots, t_p(d), x(d))$ for some small $\delta$ and $s\in[0,\delta]$. 
The idea is that if $\phi$ moves just a little bit, it still won't hit the walls at the $t_k$.

The fiber over a point $\phi\in\Phi_0^N$ is the realization of the nerve of the poset \[P:= \{t\in\mathbb{R}\vert \phi\cap\left(\mathbb{R}^{1}\times \{t_k\}\times I^{N-1}\right) = \emptyset\}\].
Lemma \ref{posetthing} shows that the nerve of this poset is contractible.
The microfibration $\epsilon$ has contractible fibers so it is a homotopy equivalence, yielding the following zigzag:
\[\Phi_1^N \xleftarrow{\simeq} \vert\vert Z_\bullet\vert\vert \xrightarrow{\simeq} B\tilde{\Phi}_0^N.\qedhere\]
\end{proof}

\begin{prop}\label{higherdelooping}
For all $N$ and $1\leq k \leq N-1$, there is a weak equivalence:
\begin{align}
\Phi_k^N \simeq \Omega \Phi_{k+1}^N,
\end{align}
\end{prop}

\begin{proof}
We are going to use Lemma \ref{semiSegal}.
The proof is very similar to that of Proposition \ref{deloopinglevelone}.
Let us construct a semi-simplicial Segal space $X_\bullet$ such that $X_1\simeq \Phi_{k}^N$ and $\vert\vert X_\bullet\vert\vert \simeq \Phi_{k+1}^N$.

Let $X_\bullet$ be the semi-simplicial space whose space of $p$-simplices is the subspace of $\Phi_{k+1}^N\times N_p\mathbb{R}$ of all $(\phi,t_0,\dots, t_p)$ such that $\phi$ does not intersect the hyperplanes $\mathbb{R}^{k+1}\times \{t_k\}\times I^{N-k-1}$, for every $0\leq k \leq p$.
The $i$-th face map forgets $t_i$.
The Segal maps
\begin{align*}
    X_n &\rightarrow X_1\times_{X_0}\dots\times_{X_0} X_1\\
    (\phi,t_0,\dots,t_n) &\mapsto ((\phi,t_0,t_1),\dots,(\phi,t_{n-1},t_n))
\end{align*}
are isomorphisms, so $X_\bullet$ is a semi-simplicial Segal space. \\
Let us show that $X_1$ deformation retracts onto $\Phi_{k}^N$. 
Define a map $l:X_1\rightarrow \Phi_{k}^N$ which restricts $(\phi,t_0,t_1)$ to $\phi\cap \left(\mathbb{R}^{k}\times (t_0,t_1) \times I^{N-k}\right)$ and then rescales $(t_0,t_1)$ to $(0,1)$ to get an element of $\Phi_{k}^N$.\\
There is another map $m:\Phi_{k}^N\rightarrow X_1$ which sends $\phi$ to $(\phi,0,1)$.
Clearly, $l\circ m = \id$.
Moreover, there is a homotopy between $m\circ l$ and the identity of $X_1$, which is obtained by pushing off whatever is outside of $[t_0,t_1]$ to infinity (which is possible by using the topology defined in Definition \ref{weirdtop}) and simultaneously rescaling $[t_0,t_1]$ to $[0,1]$.
Consequently, $X_1\simeq \Phi_{k}^N$.

There is a map
\begin{align}\label{fib}
    \vert\vert X_\bullet\vert\vert\rightarrow \Phi_{k+1}^N
\end{align}
which forgets the $(t_k)$. It is a microfibration following the same argument as in the proof of proposition \ref{deloopingleveltwo}.

The fiber of the map $\epsilon$ over a point $\phi$ is the classifying space of the topological poset $P:= \{t\in\mathbb{R}\vert \phi\cap\left(\mathbb{R}^{k+1}\times \{t_k\}\times I^{N-k-1}\right) = \emptyset\}$.
We can show that this poset is contractible using lemma \ref{posetthing}.

The map \ref{fib} is a microfibration with contractible fibers, so it is an equivalence.
By lemma \ref{semiSegal}, we get:
\[\Phi_k^N \simeq \Omega \Phi_{k+1}^N.\qedhere\] 
\end{proof}

\begin{cor}\label{loopdeloop}
    There is a weak equivalence \[\Omega B \Phi_0^N \simeq \Omega^N \Phi_N^N.\]
\end{cor} 

\section{Describing $\Phi_N^N$}

We will now give a simpler description of the space $\Phi_N^N$, which contains all the local images of elements in $\Phi_0^N$.
\begin{defi}
    Let $U_1$ be the subset of $\Phi_N^N$ which contains those $\phi$ such that there exists $r>0$ that satisfies that $\phi\cap B(0,r)$ has exactly one path. 
    This in particular implies that $\phi(0)\in B(0,r)$.\\
    Let $U_0$ be the subset of $\Phi_N^N$ which contains those $\phi$ such that there exists $r>0$ that satisfies that $\phi\cap B(0,r)$ is empty.
    We denote by $U_{01}$ the intersection $U_0\cap U_1$.
\end{defi}

\begin{lem}
    The space $\Phi_N^N$ is weakly equivalent to the homotopy pushout:
    \begin{center}
    \begin{tikzcd}
    U_{01}\arrow{r}\arrow{d} &U_1\arrow{d}\\
    U_0\arrow{r} &\Phi_N^N.
    \end{tikzcd}
    \end{center}
\end{lem}

\begin{proof}
The space $\Phi_N^N$ is the union of $U_0$ and $U_1$, and a pushout of open sets over their intersection is a homotopy pushout,
so let us prove that $U_0$ and $U_1$ are open.
Let us take $\phi\in U_0$, and $\psi$ such that $\vert\vert \phi -\psi\vert\vert < \epsilon$. (Recall that we are dealing with functions that are constant outside of a compact space so the norm is well-defined.)
There exists $r$ such that $\phi\cap B(0,r)$ is empty.
If we take $r'=r-\epsilon$, $\psi\cap B(0,r')$ is empty, so $\psi\in U_0$.
Therefore, $U_0$ is open.
In the same way, we prove that $U_1$ is open.
\end{proof}

\begin{defi}
    Let $U'_1$ be the space of pairs $(\phi, r)$ where $\phi \in U_1$ and $r$ is such that $\phi\cap B(0,r)$ has exactly one unique path closest to the origin.\\
    Let $U'_0$ be the space of pairs $(\phi, r)$ where $\phi\cap B(0,r)$ is empty.
\end{defi}
    
\begin{lem}
The projection maps $U'_1\rightarrow U_1$, and $U'_0\rightarrow U_0$ are weak equivalences.
\end{lem}
    
\begin{proof}
    Let us prove that the map $U'_1\rightarrow U_1$ is a microfibration. Consider given the outer square:
    \begin{center}
        \begin{tikzcd}[row sep = 0.4cm, column sep = 0.4cm]
        D^k\times\{0\}\arrow{dd}\arrow{dr} \arrow{rr}{A} &&U'_1\arrow{dd}{p}\\
        &D^k\times[0,\delta] \arrow[dashed]{ru}{H}\arrow{dl}\\
        D^k\times[0,1]\arrow{rr}{B} &&U_1.
        \end{tikzcd}
    \end{center}
    If $A(d) = (\phi(d),r(d)) \in U'_1$ and $B(d,s) = \phi'(d,s)$, 
    let us set $r'(d,s)= r(d) - \vert\vert \phi(d) - \phi'(d,s)\vert\vert_\infty$ on $[0;\delta]$, where $\delta$ is small enough that for $s\leq \delta$, $r'(d,s)> \vert\vert\phi'(d,s)(0)\vert\vert$.
    Then $H(d,s) = (\phi'(d,s),r'(d,s))$ is a lift in the diagram.
    The fiber over $\phi$ contains all $r$ such that $(\phi,r)\in U'_1$. 
    If $r$ and $r'$ are in the fiber, choose $r<r''<r'$.
    By definition, $\phi\cap B(0,r)$ has one path, and the same holds for $\phi\cap B(0,r')$, so $\phi\cap B(0,r'')$ as well.\\
    Therefore, the fiber of $U'_1\rightarrow U_1$ is an interval, so it is contractible.
    The map $U'_1\rightarrow U_1$ is a weak equivalence.
In the same way, the map $U'_0\rightarrow U_0$ is a weak equivalence.
\end{proof}
    
\begin{defi}
    Let $U''_{01}$ be the space $\{(\phi, r, r') \text{ where }(\phi, r)\in U'_0\text{ and }(\phi, r')\in U'_1\}$.
    \end{defi}
    
    \begin{lem}
        The map $U''_{01}\rightarrow U_{01}$ is a weak equivalence.
    \end{lem}
    
    \begin{proof}
    The projection $U''_{01}\rightarrow U_{01}$ is a microfibration as above. The fiber is the space of pairs $(r,r')$ such that $(\phi, r)$ is an element of $U_0$ and $(\phi, r')$ is an element of $U_1$, so it is a product of nonempty intervals of $\mathbb{R}$ and therefore contractible.
    \end{proof}
    
    \begin{cor}
    The space $\Phi_N^N$ is weakly equivalent to the homotopy pushout:
        \begin{center}
        \begin{tikzcd}
        U''_{01}\arrow{r}\arrow{d}&U'_1\arrow{d}\\
        U'_0\arrow{r} &\Phi_N^N
        \end{tikzcd}
        \end{center}
    \end{cor}

    \begin{lem}
        The spaces $U'_0$ and $U'_1$ are contractible.
    \end{lem}

    \begin{proof}
    The following maps
    \begin{alignat*}{4}
        f: U'_1 &\rightarrow D^{N} \qquad \qquad &g: D^{N}&\rightarrow U'_1\\
        (\phi,r)&\mapsto \frac{\phi(0)}{r}\qquad \qquad &x &\mapsto (t\mapsto (t,x),1)
    \end{alignat*}
    are homotopy inverses of each other, therefore $U'_1\simeq *$.\\
    We define a homotopy between the identity on $U'_0$ and the constant map to $(\emptyset,1)$, where $\emptyset$ is the point at infinity,
    which pushes $\phi$ to infinity along the vector $u:= \frac{\phi(0)}{\vert\vert\phi(0)\vert\vert}$, and then rescales $r$ to 1.
    \end{proof}

    \begin{lem}
    There is an equivalence $U''_{01}\simeq S^{N-1}$.
    \end{lem}
        
    \begin{proof}
    We define a map 
    \begin{align*}
    f: U''_{01} &\rightarrow S^{N-1}\\
    (\phi,r,r') &\mapsto \frac{\phi(0)}{\vert\vert\phi(0)\vert\vert}
    \end{align*}
    which is well-defined because $(\phi,r')\in U'_0$ so $\phi(0)\notin B(0,r')$ and in particular $\phi(0)\neq 0$.
    There is another map 
    \begin{align*}
    g: S^{N-1} &\rightarrow U''_{01}\\
    x &\mapsto (\phi:t\mapsto (x,t),r=2,r'=1)\\
    \end{align*}
    Notice that $f\circ g = \id$. 
    We also build a homotopy $h$ between $\id$ and $g\circ f$, which rescales $\phi$ to a constant path, and $r$ and $r'$ to $1$ and $2$.
    This concludes the proof that $U''_{01}$ and $S^{N-1}$ are equivalent.
    \end{proof}
 
\begin{cor}\label{finaldelooping}
The space $\Phi_N^N$ is equivalent to the $N$-sphere $S^N$.
\end{cor}

We will describe the loop space $\Omega\Phi_{N+1}^{N+1}$ by using the convention that $S^1$ is the one-point compactification of $\mathbb{R}$.
The spaces $(\Phi_N^N)$ assemble into a spectrum $\Phi$ with the maps
\begin{align*}
    \omega_N: \Phi_N^N&\rightarrow \Omega \Phi_{N+1}^{N+1}\\
    \phi &\mapsto (\gamma : t\mapsto \phi(-)+(0,\dots,0,t)). 
\end{align*}
We would like to have a stronger result -- we need compatibility relations between the maps $\Phi_N^N\rightarrow \Omega \Phi_{N+1}^{N+1}$ and $S^N\rightarrow \Omega S^{N+1}$.
\begin{lem}\label{spectrumeq}
    There is an equivalence between the spectrum $\Phi$ and the sphere spectrum.
\end{lem}

\begin{proof}
    Recall $U'_0$, $U'_1$, and $U''_{01}$ defined above. We denote $\Phi_N^{'N}$ the pushout of $U'_0\leftarrow U''_{01}\rightarrow U'_1$.
    It is not apparent in the notation that these spaces depend on $N$ but they do.
    The $\Phi_N^{'N}$ form a spectrum, via the maps $\Phi_{N}^{'N}\rightarrow \Omega\Phi_{N+1}^{'N+1}$ which send $(\phi,r)$ to $(\omega_N(\phi),r)$. 

    We are going to show that there is a zigzag of spectra:
    \[\Phi_\bullet^\bullet \leftarrow \Phi_\bullet^{'\bullet} \rightarrow \mathbb{S}^\bullet.\]
    We need to show that the following diagram commutes:
    \begin{center}
    \begin{tikzcd}
        \Phi_N^N\arrow{d}{\omega_N} &\Phi_N^{'N}\arrow{l}\arrow{r}\arrow{d} &S^N\arrow{d}\\
        \Omega\Phi_{N+1}^{N+1}&\Omega \Phi_{N+1}^{'N+1}\arrow{l}\arrow{r} &\Omega S^{N+1}.
    \end{tikzcd}
    \end{center}
    
    Let us check this in detail for $(\phi,r)\in U'_0$.
    The left square commutes because both vertical maps are suspensions, and both horizontal maps are projections.
    The right square commutes as well: $(\phi,r)\in U'_0$ is mapped to the basepoint of $S^N$ and then the constant path at the basepoint in $\Omega S^{N+1}$.
    Mapping first vertically sends it to the path $(\phi(-)+(0,\dots,0,t),r)$ which is also mapped to the constant path at the basepoint of $\Omega S^{N+1}$.
    We proceed in a similar fashion for $U'_1$. 

    The horizontal maps are weak homotopy equivalences, so we have a zigzag of weak equivalences of spectra, and therefore a weak equivalence $\Phi \simeq \mathbb{S}$.
\end{proof}

\section{Conclusion}
We now relate the colimit of the classifying spaces of the symmetric groups denoted by $B\Sigma_\infty$ and the monoid given by their disjoint union $\bigsqcup_{n\in\mathbb{N}} B\Sigma_n$.
Consider the map $\sigma: \bigsqcup B\Sigma_n \rightarrow \bigsqcup B\Sigma_n$ induced by the maps $B\Sigma_k\rightarrow B\Sigma_{k+1}$.

As stated in the proof of Proposition 1 in \cite{mcduffsegal}, the group completion of the monoid, denoted by $\Omega B(\bigsqcup_{n\in\mathbb{N}} B\Sigma_{n})$, is equivalent in homology to the colimit 
\[\operatorname{colim} \bigg(\bigsqcup_n B\Sigma_n\xrightarrow{\sigma} \bigsqcup_n B\Sigma_n\xrightarrow{\sigma}\dots\bigg).\]
Moreover:
\[\operatorname{colim} \bigg(\bigsqcup_n B\Sigma_n\xrightarrow{\sigma} \bigsqcup B\Sigma_n\xrightarrow{\sigma}\dots\bigg) \simeq \mathbb{Z}\times \operatorname{colim}B\Sigma_n.\]

Therefore, there is a homology equivalence.
\begin{align}\label{gpcompletion}
\mathbb{Z}\times B\Sigma_\infty \overset{H_*}{\simeq} \Omega B(\bigsqcup_{n\in\mathbb{N}} B\Sigma_n).
\end{align}

\begin{teo}
There is a homology equivalence :
\[ B\Sigma_\infty \overset{H_*}{\simeq} \Omega^\infty_0 S^\infty.\]
\end{teo}

\begin{proof}
Corollary \ref{loopdeloop} yields 
\[\Omega B \Phi_0^N \simeq \Omega^N \Phi_N^N.\]
Combining this with lemma \ref{finaldelooping}, we get: 
\[\Omega B \Phi_0^N \simeq \Omega^N S^N.\]
Now, lemma \ref{deloopinglevelone} gives us:
\[\Omega B\C_N\simeq \Omega^N S^N.\]
If we let $N$ go to infinity, we get an equivalence with the infinite loop space containing the basepoint via lemma \ref{spectrumeq}:
\[\Omega B\C_\infty \simeq \Omega^\infty S^\infty.\]
Then using proposition \ref{homeq}, we get
\[\Omega B \bigsqcup B\Sigma_n \simeq \Omega^\infty S^\infty,\]
and therefore, using equation \ref{gpcompletion}, we obtain: 
\[B\Sigma_\infty \overset{H_*}{\simeq} \Omega^\infty_0 S^\infty.\qedhere\] 
\end{proof}

\printbibliography

\end{document}